\documentclass{article}
\usepackage{cite}
\usepackage{amsmath,amssymb,amsthm}
\usepackage{titlesec,hyperref}
\usepackage{color}
\usepackage{fancyhdr}
\usepackage[margin=2.5cm]{geometry}
\usepackage{enumerate}
\usepackage[integrals]{wasysym}
\normalsize
\usepackage{bm}
\usepackage{cases}
\usepackage{graphicx}
\usepackage{caption}
\usepackage{float}
\usepackage{subfigure}
\usepackage[sort&compress,numbers]{natbib}

\newtheorem{theo}{Theorem}[section]
\newtheorem{lemm}[theo]{Lemma}

\newtheorem{prop}[theo]{Proposition}
\newtheorem{rema}[theo]{Remark}
\numberwithin{equation}{section}

\def\ud{{\rm d}}
\def\ep{\varepsilon}

\allowdisplaybreaks

\begin{document}

\title{The well-posedness, ill-posedness  and non-uniform dependence on initial data for the Fornberg-Whitham equation in Besov spaces}
\author{
    Yingying $\mbox{Guo}$ \footnote{email: guoyy35@fosu.edu.cn}\\
    $\mbox{Department}$ of Mathematics, Foshan University,\\
    Foshan, 528000, China
}
\date{}
\maketitle
\begin{abstract}
In this paper, we first establish the local well-posedness (existence, uniqueness and continuous dependence) for the Fornberg-Whitham equation in both supercritical Besov spaces $B^s_{p,r},\ s>1+\frac{1}{p},\ 1\leq p,r\leq+\infty$ and critical Besov spaces $B^{1+\frac{1}{p}}_{p,1},\ 1\leq p<+\infty$, which improves the previous work \cite{y2,ho,ht}. Then, we prove the solution is not uniformly continuous dependence on the initial data in supercritical Besov spaces $B^s_{p,r},\ s>1+\frac{1}{p},\ 1\leq p\leq+\infty,\ 1\leq r<+\infty$ and critical Besov spaces $B^{1+\frac{1}{p}}_{p,1},\ 1\leq p<+\infty$. At last, we show that the solution is ill-posed in $B^{\sigma}_{p,\infty}$ with $\sigma>3+\frac{1}{p},\ 1\leq p\leq+\infty$.  
\end{abstract}

\tableofcontents

\section{Introduction}
\par
In this paper, we consider the Cauchy problem for the following Fornberg-Whitham (FW) equation
\begin{align}\label{fw}
\left\{\begin{array}{ll}
u_{xxt}-u_{t}+\frac{9}{2}u_{x}u_{xx}+\frac{3}{2}uu_{xxx}-\frac{3}{2}uu_{x}+u_{x}=0,\ &t>0,\ x\in\mathbb{R},\\
u(0,x)=u_{0}(x),\ &x\in\mathbb{R},
\end{array}\right.
\end{align}
which was proposed by Whitham and Fornberg \cite{fw} as a model for breaking waves. Note that $(1-\partial_{xx})^{-1}f={\rm{p}}\ast f$ for any $f\in L^2$, where $\ast$ denote the convolution and ${\rm{p}}(x)=\frac{1}{2}e^{-|x|}$.
We can rewrite Eq. \eqref{fw} in non-local form
\begin{align}\label{fww}
\left\{\begin{array}{ll}
u_{t}+\frac{3}{2}uu_{x}=(1-\partial_{xx})^{-1}\partial_{x}u=\partial_{x}{\rm p}\ast u,\ &t>0,\ x\in\mathbb{R},\\
u(0,x)=u_{0}(x),\ &x\in\mathbb{R}.
\end{array}\right.
\end{align}
In this form, the FW equation was compared with the famous Korteweg–de Vries (KdV) equation \cite{kdv}
\begin{align*}
u_{t}+6uu_{x}=-\partial_{xxx}u
\end{align*}
and the classical Camassa-Holm (CH) equation \cite{ch,cl}
\begin{align*}
u_{t}+uu_{x}=-\partial_{x}(1-\partial_{xx})^{-1}\Big(u^{2}+\frac{1}{2}u^{2}_{x}\Big).
\end{align*}

The KdV equation admits solitons or solitary traveling wave solutions which maintain a constant shape and move at constant velocity. Indeed, the solitary wave solutions of the KdV equation in the non-periodic case are shown as
\begin{align*}
u(t,x)=\frac{c}{2}{\rm sech}^{2}\bigg(\frac{\sqrt{c}}{2}(x-ct)\bigg),
\end{align*}
where the constant $c$ is the wave speed. Unfortunately, the KdV equation did not have the property of wave breaking. Furthermore, it did not produce solitary waves of greatest height with a sharp peaked crest which are now called peakons.

In 1993, Camassa and Holm \cite{ch} found an integrable shallow water equation with peakon (peaked solitons) solutions which are solitons with discontinuous first derivative. The simplest one in the non-periodic case is of the form 
\begin{align*}
u(t,x)=ce^{-|x-ct|},
\end{align*}
where $c$ is a positive constant. The local well-posedness and local ill-posedness of the Cauchy problem for the CH equation in Sobolev spaces and Besov spaces have been investigated in  \cite{ce2,d1,d2,glmy,yyg,lio,liy,rb}. Moreover, the CH equation has global strong solutions, blow-up strong solutions, global weak solutions, global conservative weak solutions and dissipative weaak solutions, see \cite{c2,ce4,ce2,ce3,lio,bc1,bc2,cmo,hr1,xz1}. 
Further, the non-uniform continuity of the CH equation in Sobolev spaces and Besov spaces has been studied in many papers, see \cite{hmp,hk,hkm,lyz,lwyz}. 

It is interesting that the FW equation does not only admit solitary traveling wave solutions like the KdV equation, but also possess peakon solutions (or peaked traveling wave solutions) as the CH equation which are of the form
\begin{align*}
u(t,x)=\frac{8}{9}e^{-\frac{1}{2}|x-\frac{4}{3}t|},
\end{align*}
which were first found in \cite{fw}. A classification of other traveling wave solutions of the FW equation was presented by Yin, Tian and Fan \cite{ytf}. It's worth noting that the KdV equation and CH equation are integrable, and they possess infinitely many conserved quantities, an infinite hierarchy of quasi-local symmetries, a Lax pair and a bi-Hamiltonian structure. However, the FW equation is not integrable and the only useful conservation law we know so far is $\|u\|_{L^{2}}$. Therefore, the analysis of the FW equation would be somewhat more difficult.

Recently, the local well-posedness for \eqref{fww} in Sobolev spaces $H^{s},\ s>\frac{3}{2}$ and Besov spaces $B^{s}_{2,r},\ s>\frac{3}{2},\ 1<r<+\infty$ or $s\geq\frac{3}{2},\ r=1$ were established in \cite{ho,ht}. They also proved that the data-to-solution map is not uniformly continuous but H\"{o}lder continuous in some given topology. Furthermore, a blowup criterion for solutions was given. Later, Haziot \cite{ha}, H\"{o}rmann \cite{hor}, Wei \cite{w1,w2} and Wu et al \cite{wz} sharpened this blowup criterion and presented the sufficient conditions about the initial data to guarantee wave-breaking in finite time for the FW equation on the line and on the circle.

However, the local well-posedness for \eqref{fww} in Besov space $B_{p,r}^{s},\ s>1+\frac{1}{p},\ 1\leq p,\ r\leq+\infty$ or $s=1+\frac{1}{p},\ 1\leq p<+\infty,\ r=1$ has not been studied. The non-uniform dependence on initial data for \eqref{fww} in $B_{p,r}^{s},\ s>1+\frac{1}{p},\ 1\leq p\leq+\infty,\ 1\leq r<+\infty$ or $s=1+\frac{1}{p},\ 1\leq p<+\infty,\ r=1$ and local ill-posedness for \eqref{fww} in $B_{p,\infty}^{s},\ s>1+\frac{1}{p},\ 1\leq p\leq+\infty$ have also not been investigated yet. In the paper, following the idea of \cite{yyg,lyz,lwyz}, we aim to study the local well-posedness, local ill-posedness and non-uniform dependence on initial data for \eqref{fww} in Besov spaces.

our main results are stated as follows.
\begin{theo}\label{th}
Let $s\in\mathbb{R},\ 1\leq p,\ r \leq\infty$ and let $(s,p,r)$ satisfy the condition
\begin{equation}\label{condition}
s>1+\frac 1 p,\ 1\leq p,\ r\leq+\infty\quad\text{or}\quad s=1+\frac{1}{p},\ 1\leq p<+\infty,\ r=1.
\end{equation}
Assume $u_{0}\in B^s_{p,r}.$ Then, there exists a $T>0$ such that \eqref{fww} has a unique solution $u$ in $E^s_{p,r}(T)$ with the initial data $u_0$ and the map $u_0\mapsto u$ is continuous from any bounded subset of $B^s_{p,r}$ into $E^s_{p,r}(T)$. Moreover, for all $t\in[0,T]$, we have 
\begin{align}
\|u(t)\|_{B^s_{p,r}}\leq C\|u_0\|_{B^s_{p,r}}.\label{unformly}
\end{align}
\end{theo}
\begin{rema}
Our result covers the well-posedness
results in \cite{ho,ht}. In fact, when $p=r=2$, this corresponds to the Sobolev space $H^{s}$ where well-posedness has been shown by Yin \cite{y2} by applying Kato’s semigroup approach. Well-posedness in Sobolev spaces $H^{s}$ for $s>\frac{3}{2}$ was also shown by Holmes \cite{ho} where he utilized a Galerkin type approximation argument. When $p=2$ and $s>\frac{3}{2},\ 1<r<\infty$ or $s\geq\frac{3}{2},\ r=1$, this corresponds to the Besov space $B^{s}_{2,r}$ where well-posedness has been shown by Holmes and Thompson \cite{ht} by using some standard a priori estimates for linear transport equations. 
\end{rema}

From our well-posedness result, we are also able to demonstrate that the dependence on the initial data in Besov space $B^{s}_{p,r}$ with $s>1+\frac{1}{p},\ 1\leq p\leq+\infty,\ 1\leq r<+\infty$ or $s=1+\frac{1}{p},\ 1\leq p<+\infty,\ r=1$ is sharp, as summarized in the following theorem.
\begin{theo}\label{uniform}
Let $s\in\mathbb{R},\ 1\leq p,\ r\leq\infty$ and let $(s,p,r)$ satisfy the condition
\begin{equation}\label{cond}
s>1+\frac 1 p,\ 1\leq p\leq+\infty,\ 1\leq r<+\infty\quad\text{or}\quad s=1+\frac{1}{p},\ 1\leq p<+\infty,\ r=1.
\end{equation}
Then the solution map of problem \eqref{fww} is not uniformly continuous from any bounded subset in $B^{s}_{p,r}$ into $C([0,T];B^{s}_{p,r})$. More precisely, there exists two sequences of solutions $u^n$ and ${\rm w}^n$ with the initial data $u^n_0={\rm w}^n_0+{\rm v}_0^n$ and ${\rm w}^n_0$ such that 
\begin{align*}
\|{\rm w}^n_0\|_{B^{s}_{p,r}}\lesssim 1\qquad and \qquad \lim\limits_{n\rightarrow\infty}\|{\rm v}^n_0\|_{B^{s}_{p,r}}=0,
\end{align*}
but
\begin{align*}
\liminf\limits_{n\rightarrow\infty}\Big\|u^n-{\rm w}^n\Big\|_{B^{s}_{p,r}}\gtrsim t,	
\qquad \forall\ t\in[0,T_0],
\end{align*}
with small time $T_0\leq T$.
\end{theo}
\begin{rema}
For the non-uniform denpendence of the solutions to \eqref{fww} in Besov spaces, the key argument is to construct the initial data.
\end{rema}

Thanks to Theorem \ref{th}, for $r=+\infty$, the solution map of  \eqref{fww} is weak continuous with respect to the initial data $u_{0}\in B^{s}_{p,\infty}$ with $s>1+\frac{1}{p},\ 1\leq p\leq +\infty$. In fact, the data-to-solution map of \eqref{fww} is not continuous, i.e. the Cauchy problem of the FW equation \eqref{fww} is ill-posed in $B^\sigma_{p,\infty}$ with $\sigma>3+\frac{1}{p},\ 1\leq p\leq +\infty$, which is achieved in the following theorem.  
\begin{theo}\label{ill}
Let $\sigma>3+\frac 1 p$ with $1\leq p\leq \infty$. There exists a initial data $u_0\in B^\sigma_{p,\infty}$ and a positive constant $\ep_0$ such that the data-to-solution map $u_0\mapsto u(t)$ of \eqref{fww} satisfies
\begin{align*}
\limsup_{t\to0^+}\|u(t)-u_0\|_{B^\sigma_{p,\infty}}\geq \ep_0.
\end{align*}
\end{theo}
\begin{rema}
Theorem \ref{ill} demonstrates the ill-posedness of the FW equation in $B^\sigma_{p,\infty}$. More precisely, there exists a $u_0\in B^\sigma_{p,\infty}$ such that the corresponding solution to the FW equation that starts from $u_0$ does not converge back to $u_0$ in the sense of $B^\sigma_{p,\infty}$-norm as time goes to zero. Our key argument is to construct a initial data $u_0$.
\end{rema}

Our paper unfolds as follows. In the second section, we introduce some preliminaries which will be used in this sequel. In the third section, we establish the local well-posedness and continuous dependness of \eqref{fww} in $B^s_{p,r}$ with $s>1+\frac{1}{p},\ 1\leq p,\ r\leq+\infty$ or $s=1+\frac{1}{p},\ 1\leq p<+\infty,\ r=1$. In the fourth section, we give the non-uniform dependence on initial data for \eqref{fww} in $B_{p,r}^{s}$ with $ s>1+\frac{1}{p},\ 1\leq p\leq+\infty,\ 1\leq r<+\infty$ or $s=1+\frac{1}{p},\ 1\leq p<+\infty,\ r=1$. In the last section, by constructing a initial data $u_{0}\in B^{\sigma}_{p,\infty}$ with $\sigma>3+\frac 1 p,\ 1\leq p\leq+\infty$, we prove that the corresponding solution to \eqref{fww} starting from $u_{0}$ is discontinuous at $t=0$ in the norm of $B^{\sigma}_{p,\infty}$, which implies  the ill-posedness for \eqref{fww} in $B^{\sigma}_{p,\infty}$.

\textbf{Notation.} In the following, given a Banach space $X$, we denote its norm by $\|\cdot\|_{X}$. For $I\subset\mathbb{R}$, we denote by $C(I;X)$ the set of continuous functions on $I$ with values in $X$. Sometimes we will denote $L^q(0,T;X)$ by $L_T^qX$.

\section{Preliminaries}
\par
In this section, we first introduce the Bernstein's inequalities and some properties of the Littlewood-Paley theory and Besov spaces in \cite{book}.

\begin{prop}[Bernstein's inequalities, See \cite{book}]\label{bern}
Let $\mathfrak{B}$ be a ball and $\mathfrak{C}$ be an annulus. A constant $C>0$ exists such that for all $k\in\mathbb{N},\ 1\leq  p\leq q\leq\infty$, and any function $f\in L^{p}(\mathbb{R}^{d})$, we have
\begin{align*}
&{\rm Supp}{\widehat{f}}\subset\lambda\mathfrak{B}\Rightarrow\|D^{k}f\|_{L^{q}}=\sup\limits_{|\alpha|=k}\|\partial^{\alpha}f\|_{L^{q}}\leq C^{k+1}\lambda^{k+d(\frac1p-\frac1q)}\|f\|_{L^{p}},\\
&{\rm Supp}{\widehat{f}}\subset\lambda\mathfrak{C}\Rightarrow C^{-k-1}\lambda^{k}\|f\|_{L^{p}}\leq\|D^{k}f\|_{L^{p}}\leq C^{k+1}\lambda^{k}\|f\|_{L^{p}}.
\end{align*}
\end{prop}

Let $\chi: \mathbb{R}\rightarrow[0,1]$ be a radical, smooth, and even function which is suppported in $\mathcal{B}=\{\xi:|\xi|\leq\frac 4 3\}$.
Let $\varphi:\mathbb{R}\rightarrow[0,1]$ be a radical, smooth function which is suppported in $\mathcal{C}=\{\xi:\frac 3 4\leq|\xi|\leq\frac 8 3\}$.

Denote $\mathcal{F}$ and $\mathcal{F}^{-1}$ by the Fourier transform and the Fourier inverse transform respectively as follows:
\begin{align*}
&\mathcal{F}u(\xi)=\hat{u}(\xi)=\int_{\mathbb{R}^{d}}e^{-ix\xi}u(x){\ud}x,\\
&u(x)=\big(\mathcal{F}^{-1}\hat{u}\big)(x)=\frac{1}{2\pi}\int_{\mathbb{R}^{d}}e^{ix\xi}\hat{u}(\xi){\ud}\xi.	
\end{align*}

For any $u\in\mathcal{S}'(\mathbb{R}^d)$ and all $j\in\mathbb{Z}$, define
$\Delta_j u=0$ for $j\leq -2$; $\Delta_{-1} u=\mathcal{F}^{-1}(\chi\mathcal{F}u)$; $\Delta_j u=\mathcal{F}^{-1}(\varphi(2^{-j}\cdot)\mathcal{F}u)$ for $j\geq 0$; and $S_j u=\sum\limits_{j'<j}\Delta_{j'}u$.

Let $s\in\mathbb{R},\ 1\leq p,\ r\leq\infty.$ We define the nonhomogeneous Besov space $B^s_{p,r}(\mathbb{R}^d)$
$$  B^s_{p,r}=B^s_{p,r}(\mathbb{R}^d)=\Big\{u\in S'(\mathbb{R}^d):\|u\|_{B^s_{p,r}}=\big\|(2^{js}\|\Delta_j u\|_{L^p})_j \big\|_{l^r(\mathbb{Z})}<\infty\Big\}.$$

The corresponding nonhomogeneous Sobolev space $H^{s}(\mathbb{R}^d)$ is
$$H^{s}=H^{s}(\mathbb{R}^d)=\Big\{u\in S'(\mathbb{R}^d):\ u\in L^2_{loc}(\mathbb{R}^d),\ \|u\|^2_{H^s}=\int_{\mathbb{R}^d}(1+|\xi|^2)^s|\mathcal{F}u(\xi)|^2{\ud}\xi<\infty\Big\}. $$

We introduce a function space, which will be used in the following.
\begin{equation*}
E^s_{p,r}(T)\triangleq \left\{\begin{array}{ll}
C([0,T];B^s_{p,r})\cap C^1([0,T];B^{s-1}_{p,r}), & \text{if}\ r<\infty,  \\
C_w([0,T];B^s_{p,\infty})\cap C^{0,1}([0,T];B^{s-1}_{p,\infty}), & \text{if}\ r=\infty.
\end{array}\right.
\end{equation*}

Then, we recall some properties about the Besov spaces.
\begin{prop}[See \cite{book}]\label{Besov}
Let $s\in\mathbb{R},\ 1\leq p,\ p_1,\ p_2,\ r,\ r_1,\ r_2\leq\infty.$  \\
{\rm(1)} $B^s_{p,r}$ is a Banach space, and is continuously embedded in $\mathcal{S}'$. \\
{\rm(2)} If $r<\infty$, then $\lim\limits_{j\rightarrow\infty}\|S_j u-u\|_{B^s_{p,r}}=0$. If $p,\ r<\infty$, then $C_0^{\infty}$ is dense in $B^s_{p,r}$. \\
{\rm(3)} If $p_1\leq p_2$ and $r_1\leq r_2$, then $ B^s_{p_1,r_1}\hookrightarrow B^{s-(\frac d {p_1}-\frac d {p_2})}_{p_2,r_2}. $
If $s_1<s_2$, then the embedding $B^{s_2}_{p,r_2}\hookrightarrow B^{s_1}_{p,r_1}$ is locally compact. \\
{\rm(4)} $B^s_{p,r}\hookrightarrow L^{\infty} \Leftrightarrow s>\frac d p\ \text{or}\ s=\frac d p,\ r=1$. \\
{\rm(5)} Fatou property: if $(u_n)_{n\in\mathbb{N}}$ is a bounded sequence in $B^s_{p,r}$, then an element $u\in B^s_{p,r}$ and a subsequence $(u_{n_k})_{k\in\mathbb{N}}$ exist such that
$$ \lim_{k\rightarrow\infty}u_{n_k}=u\ \text{in}\ \mathcal{S}'\quad \text{and}\quad \|u\|_{B^s_{p,r}}\leq C\liminf_{k\rightarrow\infty}\|u_{n_k}\|_{B^s_{p,r}}. $$
{\rm(6)} Let $m\in\mathbb{R}$ and $f$ be a $S^m$-mutiplier $($i.e. f is a smooth function and satisfies that $\forall\ \alpha\in\mathbb{N}^d$,
$\exists\ C=C(\alpha)$ such that $|\partial^{\alpha}f(\xi)|\leq C(1+|\xi|)^{m-|\alpha|},\ \forall\ \xi\in\mathbb{R}^d)$.
Then the operator $f(D)=\mathcal{F}^{-1}(f\mathcal{F})$ is continuous from $B^s_{p,r}$ to $B^{s-m}_{p,r}$.
\end{prop}
\begin{prop}[See \cite{book}]\label{duiou}
Let $s\in\mathbb{R},\ 1\leq p,\ r\leq\infty.$
\begin{equation*}\left\{
\begin{array}{l}
B^s_{p,r}\times B^{-s}_{p',r'}\longrightarrow\mathbb{R},  \\
(u,\phi)\longmapsto \sum\limits_{|j-j'|\leq 1}\langle \Delta_j u,\Delta_{j'}\phi\rangle,
\end{array}\right.
\end{equation*}
defines a continuous bilinear functional on $B^s_{p,r}\times B^{-s}_{p',r'}$. Denote by $Q^{-s}_{p',r'}$ the set of functions $\phi$ in $\mathcal{S}'$ such that
$\|\phi\|_{B^{-s}_{p',r'}}\leq 1$. If $u$ is in $\mathcal{S}'$, then we have
$$\|u\|_{B^s_{p,r}}\leq C\sup_{\phi\in Q^{-s}_{p',r'}}\langle u,\phi\rangle.$$
\end{prop}

We next give some crucial interpolation inequalities.
\begin{prop}[See \cite{book}]\label{prop}
{\rm(1)} If $s_1<s_2,\ \lambda\in (0,1)$ and $(p,r)\in[1,\infty]^2,$ then we have
\begin{align*}
&\|u\|_{B^{\lambda s_1+(1-\lambda)s_2}_{p,r}}\leq \|u\|_{B^{s_1}_{p,r}}^{\lambda}\|u\|_{B^{s_2}_{p,r}}^{1-\lambda},\\
&\|u\|_{B^{\lambda s_1+(1-\lambda)s_2}_{p,1}}\leq\frac{C}{s_2-s_1}\Big(\frac{1}{\lambda}+\frac{1}{1-\lambda}\Big) \|u\|_{B^{s_1}_{p,\infty}}^{\lambda}\|u\|_{B^{s_2}_{p,\infty}}^{1-\lambda}.
\end{align*}
{\rm(2)} If $s\in\mathbb{R},\ 1\leq p\leq\infty,\ \varepsilon>0,$ a constant $C=C(\varepsilon)$ exists such that
$$ \|u\|_{B^s_{p,1}}\leq C\|u\|_{B^s_{p,\infty}}\ln\Big(e+\frac {\|u\|_{B^{s+\varepsilon}_{p,\infty}}}{\|u\|_{B^s_{p,\infty}}}\Big). $$
\end{prop}

The 1-D Moser-type estimates are provided as follows.
\begin{lemm}[See \cite{book}]\label{product} 
The following estimates hold:\\
{\rm(1)} For any $s>0$ and any $p,\ r$ in $[1,\infty]$, the space $L^{\infty}\cap B^s_{p,r}$ is an algebra and a constant $C=C(s)$ exists such that
\begin{align*}
\|uv\|_{B^s_{p,r}}\leq C(\|u\|_{L^{\infty}}\|v\|_{B^s_{p,r}}+\|u\|_{B^s_{p,r}}\|v\|_{L^{\infty}}).
\end{align*}
{\rm(2)} If $1\leq p,\ r\leq \infty,\ s_1\leq s_2,\ s_2>\frac{1}{p}\ (s_2 \geq \frac{1}{p}\ \text{if}\ r=1)$ and $s_1+s_2>\max(0, \frac{2}{p}-1),$ there exists $C=C(s_1,s_2,p,r)$ such that
$$ \|uv\|_{B^{s_1}_{p,r}}\leq C\|u\|_{B^{s_1}_{p,r}}\|v\|_{B^{s_2}_{p,r}}. $$
\end{lemm}

Here is the useful Gronwall lemma.
\begin{lemm}[See \cite{book}]\label{gwl}
Let $q(t),\ a(t)\in C^1([0,T])$ and $ q(t),\ a(t)>0$. Let $b(t)$ is a continuous function on $[0,T].$ Suppose that, for all $t\in [0,T],$
$$\frac{1}{2}\frac{{\ud}}{{\ud}t}q^2(t)\leq a(t)q(t)+b(t)q^2(t).$$
Then for any time $t$ in $[0,T],$ we have
$$q(t)\leq q(0)\exp\int_0^t b(\tau){\ud}\tau+\int_0^t a(\tau)\exp\big(\int_{\tau}^tb(t'){\ud}t'\big){\ud}\tau.$$
\end{lemm}


In the paper, we also need some estimates for the following Cauchy problem of 1-D transport equation:
\begin{equation}\label{transport}
\left\{\begin{array}{l}
\partial_{t}f+v\partial_{x}f=g, \\
f(0,x)=f_0(x).
\end{array}\right.
\end{equation}

\begin{lemm}[See \cite{book}]\label{existence}
Let $1\leq p\leq\infty,\ 1\leq r\leq\infty$ and $\theta> -\min(\frac 1 {p}, \frac 1 {p'}).$ Suppose $f_0\in B^{\theta}_{p,r},$ $g\in L^1(0,T;B^{\theta}_{p,r}),$ and $v\in L^\rho(0,T;B^{-M}_{\infty,\infty})$ for some $\rho>1$ and $M>0,$ and
\begin{align*}
\begin{array}{ll}
\partial_{x}v\in L^1(0,T;B^{\frac 1 {p}}_{p,\infty}\cap L^{\infty}), &\ \text{if}\ \theta<1+\frac 1 {p}, \\
\partial_{x}v\in L^1(0,T;B^{\theta}_{p,r}),\ &\text{if}\ \theta=1+\frac{1}{p},\ r>1, \\
\partial_{x}v\in L^1(0,T;B^{\theta-1}_{p,r}), &\ \text{if}\ \theta>1+\frac{1}{p}\ (or\ \theta=1+\frac 1 {p},\ r=1).
\end{array}	
\end{align*}
Then the problem \eqref{transport} has a unique solution $f$ in 
\begin{itemize}
\item [-] the space $C([0,T];B^{\theta}_{p,r}),$ if $r<\infty,$
\item [-] the space $\big(\bigcap_{{\theta}'<\theta}C([0,T];B^{{\theta}'}_{p,\infty})\big)\bigcap C_w([0,T];B^{\theta}_{p,\infty}),$ if $r=\infty.$
\end{itemize}
\end{lemm}

\begin{lemm}[See \cite{book,liy}]\label{priori estimate}
Let $1\leq p,\ r\leq\infty$ and $\theta>-\min(\frac{1}{p},\frac{1}{p'}).$
There exists a constant $C$ such that for all solutions $f\in L^{\infty}(0,T;B^{\theta}_{p,r})$ of \eqref{transport} with initial data $f_0$ in $B^{\theta}_{p,r}$ and $g$ in $L^1(0,T;B^{\theta}_{p,r}),$ we have, for a.e. $t\in[0,T],$
$$ \|f(t)\|_{B^{\theta}_{p,r}}\leq \|f_0\|_{B^{\theta}_{p,r}}+\int_0^t\|g(t')\|_{B^{\theta}_{p,r}}{\ud}t'+\int_0^t V'(t')\|f(t')\|_{B^{\theta}_{p,r}}{\ud}t' $$
or
$$ \|f(t)\|_{B^{\theta}_{p,r}}\leq e^{CV(t)}\Big(\|f_0\|_{B^{\theta}_{p,r}}+\int_0^t e^{-CV(t')}\|g(t')\|_{B^{\theta}_{p,r}}{\ud}t'\Big) $$
with
\begin{equation*}
V'(t)=\left\{\begin{array}{ll}
\|\partial_{x}v(t)\|_{B^{\frac 1 p}_{p,\infty}\cap L^{\infty}},\ &\text{if}\ \theta<1+\frac{1}{p}, \\
\|\partial_{x}v(t)\|_{B^{\theta}_{p,r}},\ &\text{if}\ \theta=1+\frac{1}{p},\ r>1, \\
\|\partial_{x}v(t)\|_{B^{\theta-1}_{p,r}},\ &\text{if}\ \theta>1+\frac{1}{p}\ (\text{or}\ \theta=1+\frac{1}{p},\ r=1).
\end{array}\right.
\end{equation*}
If $\theta>0$, then there exists a constant $C=C(p,r,\theta)$ such that the following statement holds
\begin{align*}
\|f(t)\|_{B^{\theta}_{p,r}}\leq \|f_0\|_{B^{\theta}_{p,r}}+\int_0^t\|g(\tau)\|_{B^{\theta}_{p,r}}{\ud}\tau+C\int_0^t \Big(\|f(\tau)\|_{B^{\theta}_{p,r}}\|\partial_{x}v(\tau)\|_{L^{\infty}}+\|\partial_{x}v(\tau)\|_{B^{\theta-1}_{p,r}}\|\partial_{x}f(\tau)\|_{L^{\infty}}\Big){\ud}\tau. 	
\end{align*}
In particular, if $f=av+b,\ a,\ b\in\mathbb{R},$ then for all $\theta>0,$ $V'(t)=\|\partial_{x}v(t)\|_{L^{\infty}}.$
\end{lemm}

\section{The local well-posedness in both supercritical and critical Besov spaces}
\par
In this section, we present the local well-posedness for the Cauchy problem \eqref{fww} in both supercritical Besov spaces $B^{s}_{p,r}$ with $ s>1+\frac{1}{p},\ 1\leq p,\ r\leq+\infty$ and critical Besov spaces $B^{1+\frac{1}{p}}_{p,1}$ with $1\leq p<+\infty$, which is different from the previous work \cite{ho,ht,y2}.

\begin{proof}[{\rm\textbf{The proof of Theorem \ref{th}:}}]
we divide four steps to prove it.\\

\textbf{Step 1. We will structure  a family of approximate solution sequences by iterative scheme.}
	
Assuming that $u^0=0$, we define by induction a sequence $\{u^n\}_{n\in\mathbb{N}}$ of smooth functions by solving the following linear transport equation:
\begin{equation}\label{aa}
\left\{\begin{array}{l}
\partial_tu^{n+1}+\frac{3}{2}u^n\partial_xu^{n+1}=(1-\partial_{xx})^{-1}\partial_{x}u^{n}=\partial_{x}{\rm p}\ast{u^{n}}, \\
u^{n+1}|_{t=0}=u_0.
\end{array}\right.
\end{equation}
Assume that $\{u^n\}_{n\in\mathbb{N}}$ belongs to $E^{s}_{p,r}(T)$ for all $T>0$. We know from Lemma \ref{product} (1) that $B^{s}_{p,r},\ B^{s-1}_{p,r}$ are algebras and the embedding $B^{s}_{p,r}\hookrightarrow B^{s-1}_{p,r}\hookrightarrow L^\infty$ holds. Note that the operator $\partial_{x}{\rm p}\ast$ is a $S^{-1}$-mutiplier. Then, we have
\begin{align*}
\|(1-\partial_{xx})^{-1}\partial_{x}u^{n}\|_{B^{s}_{p,r}}=\|\partial_{x}{\rm p}\ast{u^{n}}\|_{B^{s}_{p,r}}\leq C\|u^n\|_{B^{s}_{p,r}}.
\end{align*}	
From Lemma \ref{existence}, we know \eqref{aa} has a global solution $u^{n+1}\in E^{s}_{p,r}(T)$.
Thanks to Lemma \ref{priori estimate}, we infer that
\begin{align}
\|u^{n+1}(t)\|_{B^{s}_{p,r}}
\leq e^{C\int_0^t\|u^n(t')\|_{B^{s}_{p,r}}{\ud}t'}\bigg(\|u_0\|_{B^{s}_{p,r}}+C\int_0^te^{-C\int_0^\tau\|u^n(\tau')\|_{B^{s}_{p,r}}{\ud}\tau'}\|u^n\|_{B^{s}_{p,r}}{\ud}\tau\bigg).\label{dd}
\end{align}
Fix a $T>0$ such that
$T<\min\Big\{\frac{1}{2C\|u_{0}\|_{B^{s}_{p,r}}},\frac{\ln2}{C}\Big\}$ and suppose by induction that
\begin{align}	
\|u^{n}(t)\|_{B^{s}_{p,r}}
\leq\frac{2\|u_0\|_{B^{s}_{p,r}}}{1-2C\|u_0\|_{B^{s}_{p,r}}t},\ \ \ \ \ \ \forall\ t\in[0,T].\label{hh}
\end{align}
Plugging \eqref{hh} into \eqref{dd} yields
\begin{align*}
\|u^{n+1}(t)\|_{B^{s}_{p,r}}\leq&\|u_0\|_{B^{s}_{p,r}}e^{C\int_{0}^{t}\|u^{n}\|_{B^{s}_{p,r}}{\ud}\tau}\times e^{Ct}\\
\leq&\frac{2\|u_0\|_{B^{s}_{p,r}}}{1-2C\|u_0\|_{B^{s}_{p,r}}t}, \qquad\quad \forall\ t\in[0,T].
\end{align*}
Therefore, $\{u^n\}_{n\in\mathbb{N}}$ is uniformly bounded in $L^{\infty}\big(0,T;B^{s}_{p,r}\big)$. Consequently, the right side of Eq. \eqref{aa} is bounded in $L^{\infty}\big(0,T;B^{s}_{p,r}\big)$, which follows that $\{u^n\}_{n\in\mathbb{N}}$ is uniformly bounded in $E^{s}_{p,r}(T)$.\\ 
	
\textbf{Step 2. We shall prove the sequence of approximate solutions $\{u^n\}_{n\in\mathbb{N}}$ is a Cauchy sequence and the limit $u$ is indeed a solution to \eqref{fww}.}

We will show that $\{u^n\}_{n\in\mathbb{N}}$ is a Cauchy sequence in $C\big([0,T];B^{s-1}_{p,r}\big)$. In fact, for any $n,\ m\in\mathbb{N}$, we see
\begin{equation}\label{mn}
\left\{\begin{array}{l}
\partial_t(u^{n+m+1}-u^{n+1})+\frac{3}{2}u^{m+n}\partial_x(u^{n+m+1}-u^{n+1})=\frac{3}{2}(u^{n}-u^{n+m})\partial_xu^{n+1}+\partial_{x}{\rm p}\ast(u^{n+m}-u^{n}), \\
(u^{n+m+1}-u^{n+1})|_{t=0}=0.
\end{array}\right.
\end{equation}
Thanks to Lemmas \ref{existence}--\ref{priori estimate} and the uniform boundedness of $u^{n}$, we see for any $t\in[0,T]$,
\begin{align}
\|(u^{n+m+1}-u^{n+1})(t)\|_{B^{s-1}_{p,r}}\leq& Ce^{C\int_0^t\|u^{n+m}(t')\|_{B^{s}_{p,r}}{\ud}t'}\int_0^te^{-C\int_0^\tau\|u^{n+m}(\tau')\|_{B^{s}_{p,r}}{\ud}\tau'}\|(u^{n+m}-u^n)(\tau)\|_{B^{s-1}_{p,r}}{\ud}\tau\notag\\
\leq& C_{T}\int_0^t\|(u^{n+m}-u^n)(\tau)\|_{B^{s-1}_{p,r}}{\ud}\tau.\label{umn}
\end{align}
Thereby, taking advantage of the induction, we known 
\begin{align*}
\|u^{n+m+1}-u^{n+1}\|_{L^{\infty}_{T}B^{s-1}_{p,r}}\leq\frac{(TC_{T})^{n+1}}{(n+1)!}\|u^{m}\|_{B^{s}_{p,r}}\leq C_{T}2^{-n}
\end{align*}
which implies that $\{u^n\}_{n\in\mathbb{N}}$ is a Cauchy sequence in $C\big([0,T];B^{s-1}_{p,r}\big)$. Hence, $\{u^n\}_{n\in\mathbb{N}}$ converges to some limit function $u\in C\big([0,T];B^{s-1}_{p,r}\big)$.

We next need to verify that the limit $u$ indeed belongs to $E^{s}_{p,r}(T)$ and satisfies \eqref{fww}. Since $u^n$ is uniformly bounded in $L^\infty\big(0,T;B^{s}_{p,r}\big)$, we can deduce that $u\in L^\infty\big(0,T;B^{s}_{p,r}\big)$ by the Fatou property for Besov space. Thanks to 
\begin{align*}
u^n\rightarrow u \qquad \text{in $C\big([0,T];B^{s-1}_{p,r}\big)$}	
\end{align*} 
and the interpolation inequality, we get
\begin{align*}
u^n\rightarrow u \qquad \text{in $C\big([0,T];B^{s'}_{p,r}\big)$ for any $s'<s.$}	
\end{align*} 
It is a routine method to pass to the limit in Eq. \eqref{aa} and to deduce that $u$ is solution of \eqref{fww}.
Note that the right side of Eq. \eqref{fww} belongs to $L^{\infty}\big(0,T;B^{s}_{p,r}\big)$. Thus, according to Lemma \ref{existence}, we see that $u$ is in $C\big([0,T];B^{s}_{p,r}\big)$ (resp., $C_{w}\big([0,T];B^{s}_{p,r}\big)$) if $r<\infty$ (resp., $r=\infty$). Using Eq. \eqref{fww} again, we get $\partial_{t}u$  belongs to $C\big([0,T];B^{s-1}_{p,r}\big)$ if $r<\infty$, or belongs to $L^{\infty}(0,T;B^{s-1}_{p,\infty})$ if $r=\infty$. In conclusion, $u$ belongs to $E^s_{p,r}(T)$.\\
	
\textbf{Step 3. Uniqueness.}

Assume $u^{1},\ u^{2}$ are two solutions of \eqref{fww} with initial data $u^{1}_{0},\ u^{2}_{0}$, then $u^{12}:=u^{1}-u^{2}$ satisfies
\begin{equation}\label{u12}
\left\{\begin{array}{l}
\partial_tu^{12}+\frac{3}{2}u^{1}\partial_xu^{12}=-\frac{3}{2}u^{12}\partial_xu^{2}+\partial_{x}{\rm p}\ast u^{12}, \\
u^{12}|_{t=0}=u^{1}_{0}-u^{2}_{0}.
\end{array}\right.
\end{equation}
Similar to \eqref{umn}, we see for all $t\in[0,T]$,
\begin{align*}
\|u^{12}(t)\|_{B^{s-1}_{p,r}}\leq& Ce^{C\int_0^t\|u^{1}(t')\|_{B^{s}_{p,r}}{\ud}t'}\Big(\|u^{1}_{0}-u^{2}_{0}\|_{B^{s-1}_{p,r}}+\int_0^te^{-C\int_0^\tau\|u^{1}(\tau')\|_{B^{s}_{p,r}}{\ud}\tau'}\|u^{12}(\tau)\|_{B^{s-1}_{p,r}}{\ud}\tau\Big)\\
\leq& C\|u^{1}_{0}-u^{2}_{0}\|_{B^{s-1}_{p,r}}.
\end{align*}
So if $u^{1}_{0}=u^{2}_{0}$, we can  immediately obtain the uniqueness.\\
	
\textbf{Step 4. The continuous dependence.}
	
Let $u^n_0\rightarrow u^{\infty}_0$ in $B^{s}_{p,r}$ for any $n\in\mathbb{N}$. Then we have $\partial_{x}u^n_0\rightarrow \partial_{x}u^{\infty}_0$ in $B^{s-1}_{p,r}$. Denote  $\overline{\mathbb{N}}:=\mathbb{N}\cup{\infty}$ and $u^n$ the solution to \eqref{fww} with the initial value $u^n_0$ for any $n\in\overline{\mathbb{N}}$.

\textbf{Case 1.} For $r<\infty$, owing to \textbf{Step 1--Step 2}, we know that for all $n\in\overline{\mathbb{N}}$,
\begin{align*}
u^n\in C\big([0,T];B^{s}_{p,r}\big),\quad\partial_{x}u^n\in C\big([0,T];B^{s-1}_{p,r}\big),\quad  \|u^n\|_{L^\infty(0,T;B^{s}_{p,r})}\leq C\|u^n_0\|_{B^{s}_{p,r}}
\end{align*}
and for any $n\in\mathbb{N}$,
\begin{align}
u^n\rightarrow u^\infty\qquad\text{in}\quad C\big([0,T];B^{s-1}_{p,r}\big).\label{untend}
\end{align}
Next, we just need to prove $\partial_{x}u^n\rightarrow \partial_{x}u^\infty$ in $C\big([0,T];B^{s-1}_{p,r}\big)$. Before that, let's present a useful lemma.
\begin{lemm}[See \cite{book,liy}]\label{continuous}
Let $s\in\mathbb{R},\ 1\leq p\leq\infty,\ 1\leq r<\infty$ and let $s>1+\frac{1}{p}$ or $s=1+\frac{1}{p},\ 1\leq p<\infty,\ r=1$. For $n\in\overline{\mathbb{N}}$, denote $a^n$ by the solution of 
\begin{equation}\label{an}
\left\{\begin{array}{l}
\partial_{t}a^n+A^n\partial_{x}a^n=F,\\
a^n(0,x)=a_0(x),
\end{array}\right.
\end{equation} 	
with $F\in L^\infty\big(0,T;B^{s-1}_{p,r}\big)$, $a_0\in B^{s-1}_{p,r}$. Assume that $$\sup\limits_{n\in\overline{\mathbb{N}}}\|A^{n}(t)\|_{B^{s}_{p,r}}\leq\alpha(t) \quad\text{for some $\alpha\in L^{1}(0,T)$}$$ and  $A^n\rightarrow A^\infty$ in $L^1\big(0,T;B^{s-1}_{p,r}\big)$. Then the sequence $\{a^n\}_{n\in\mathbb{N}}$ converges to $a^\infty$ in $C\big([0,T];B^{s-1}_{p,r}\big)$.
\end{lemm}
We continue proving $\partial_{x}u^n\rightarrow \partial_{x}u^\infty$ in $C\big([0,T];B^{s-1}_{p,r}\big)$. Split $\partial_{x}u^n$ into $w^n+z^n$ with $(w^n,z^n)$ satisfying 
\begin{equation*}
\left\{\begin{array}{l}
\partial_{t}w^n+\frac{3}{2}u^n\partial_{x}w^n=F(u^{\infty}),\\
w^n(0,x)=\partial_{x}u^\infty_0
\end{array}\right.
\end{equation*} 
and 
\begin{equation*}
\left\{\begin{array}{l}
\partial_{t}z^n+\frac{3}{2}u^n\partial_{x}z^n=F(u^{n})-F(u^{\infty}), \\
z^n(0,x)=\partial_{x}u^n_0-\partial_{x}u^\infty_0.
\end{array}\right.
\end{equation*}
where 
\begin{align*}
&F(u^{n})=(1-\partial_{xx})^{-1}u^{n}-u^{n}-\frac{3}{2}\big(u^{n}_{x}\big)^{2}\quad\text{ for any $n\in\overline{\mathbb{N}}$}.
\end{align*}
Since $\{u^n\}_{n\in\overline{\mathbb{N}}}$ is bounded in $L^{\infty}\big(0,T;B^{s}_{p,r}\big)$, then $\{u^n_{x}\}_{n\in\overline{\mathbb{N}}}$  and $\{F(u^{n})\}_{n\in\overline{\mathbb{N}}}$ are bounded in $L^{\infty}\big(0,T;B^{s-1}_{p,r}\big)$. Notice that $u^n\rightarrow u^\infty$ in $C\big([0,T];B^{s-1}_{p,r}\big)$. Lemma \ref{continuous} thus ensures that
\begin{align}
w^n\rightarrow w^\infty\qquad \text{in}\quad  C\big([0,T];B^{s-1}_{p,r}\big).\label{winfty}
\end{align} 
Next, according to Lemmas \ref{existence}--\ref{priori estimate}, we obtain $z^\infty=0$ for any $t\in[0,T]$. Noting again that  $\{u^n\}_{n\in\overline{\mathbb{N}}}$ is bounded in $L^{\infty}\big(0,T;B^{s}_{p,r}\big)$, we have 
\begin{align*}
\|F(u^{n})-F^\infty(u^{\infty})\|_{B^{s-1}_{p,r}}\leq& C\left(\|u^n-u^\infty\|_{B^{s-1}_{p,r}}+\|\partial_{x}u^n-\partial_{x}u^\infty\|_{B^{s-1}_{p,r}}\right)\\
\leq& C\left(\|u^n-u^\infty\|_{B^{s-1}_{p,r}}+\|w^n-w^\infty\|_{B^{s-1}_{p,r}}+\|z^n-z^{\infty}\|_{B^{s-1}_{p,r}}\right)\\
\leq& C\left(\|u^n-u^\infty\|_{B^{s-1}_{p,r}}+\|w^n-w^\infty\|_{B^{s-1}_{p,r}}+\|z^n\|_{B^{s-1}_{p,r}}\right).
\end{align*}
It follows that for all $n\in\mathbb{N}$
\begin{align}
\|z^n(t)\|_{B^{s-1}_{p,r}}
\leq& e^{C\int_0^t\|u^n(t')\|_{B^{s}_{p,r}}{\ud}t'}\left(\|\partial_{x}u^n_0-\partial_{x}u^\infty_0\|_{B^{s-1}_{p,r}}+\int_0^t\|F(u^{n})-F(u^{\infty})\|_{B^{s-1}_{p,r}}{\ud}\tau\right)\notag\\
\leq&C\left(\|\partial_{x}u^n_0-\partial_{x}u^\infty_0\|_{B^{s-1}_{p,r}}+\int_0^t\|u^n-u^\infty\|_{B^{s-1}_{p,r}}+\|w^n-w^\infty\|_{B^{s-1}_{p,r}}+\|z^n\|_{B^{s-1}_{p,r}}{\ud}\tau\right).\label{zn}
\end{align}
Using the facts that
\begin{itemize}
\item [-]~$\partial_{x}u^n_0$ tends to $\partial_{x}u^\infty_0$ in $B^{s-1}_{p,r}$;
\item [-]~$u^n$ tends to $u^{\infty}$ in $C\big([0,T];B^{s-1}_{p,r}\big)$;
\item [-]~$w^n$ tends to $w^{\infty}$ in $C\big([0,T];B^{s-1}_{p,r}\big)$,
\end{itemize}
and then applying the Gronwall lemma, we conclude that $z^n$ tends to $0$ in $C\big([0,T];B^{s-1}_{p,r}\big)$.\\
Hence,
\begin{align*}
\|\partial_{x}u^n-\partial_{x}u^\infty\|_{L^\infty(0,T;B^{s-1}_{p,r})}\leq&\|w^n-w^\infty\|_{L^\infty(0,T;B^{s-1}_{p,r})}+\|z^n-z^\infty\|_{L^\infty(0,T;B^{s-1}_{p,r})}\\
\leq&\|w^n-w^\infty\|_{L^\infty(0,T;B^{s-1}_{p,r})}+\|z^n\|_{L^\infty(0,T;B^{s-1}_{p,r})}\qquad\quad\rightarrow 0\quad\text{as\ $n\rightarrow\infty,$}
\end{align*}
that is 
\begin{align}
\partial_{x}u^n\rightarrow \partial_{x}u^\infty\qquad\text{in}\quad C\big([0,T];B^{s-1}_{p,r}\big).\label{us1}
\end{align}
Combining \eqref{untend} with \eqref{us1}, we see $u^{n}\rightarrow u^{\infty}$ in $C([0,T];B^{s}_{p,r})$. 

\textbf{Case 2.} For $r=\infty$, noting that $B^{s}_{p,\infty}\hookrightarrow B^{s'}_{p,1}$ for any $s'<s$, we have $u^n\rightarrow u^\infty$ in $C\big([0,T];B^{s-1}_{p,1}\big)\hookrightarrow C\big([0,T];B^{s-1}_{p,\infty}\big)$. For fixed $\phi\in B^{1-s}_{p',1}$, we write
\begin{align*}
\big\langle u^{n}(t)-u^{\infty}(t),\phi\big\rangle=&\big\langle S_{j}\big(u^{n}(t)-u^{\infty}(t)\big),\phi\big\rangle+\big\langle \big(Id-S_{j}\big)\big(u^{n}(t)-u^{\infty}(t)\big),\phi\big\rangle\\
=&\big\langle u^{n}(t)-u^{\infty}(t),S_{j}\phi\big\rangle+\big\langle u^{n}(t)-u^{\infty}(t),\big(Id-S_{j}\big)\phi\big\rangle
\end{align*}
By means of Proposition \ref{duiou}, we see
\begin{align}
&\big|\big\langle u^{n}(t)-u^{\infty}(t),S_{j}\phi\big\rangle\big|\leq C\|u^{n}-u^{\infty}\|_{L^{\infty}(0,T;B^{s-1}_{p,\infty})}\|S_{j}\phi\|_{B^{1-s}_{p',1}},\label{s-1}\\
&\big|\big\langle u^{n}(t)-u^{\infty}(t),\big(Id-S_{j}\big)\phi\big\rangle\big|\leq C\|\big(Id-S_{j}\big)\phi\|_{B^{1-s}_{p',1}}.\label{s-2}
\end{align}
Thanks to $u^n\rightarrow u^\infty$ in $C\big([0,T];B^{s-1}_{p,\infty}\big)$ and  $S_{j}\phi\rightarrow\phi$ in $B^{1-s}_{p',1}$, we know that for fixed $j\in\mathbb{N}$, when $n\rightarrow\infty$, \eqref{s-1} tends to $0$. And then letting $j\rightarrow\infty$, \eqref{s-2} goes to $0$. Therefore, $\big\langle u^{n}(t)-u^{\infty}(t),\phi\big\rangle$ tends to $0$, i.e. $u^{n}\rightarrow u^{\infty}$ in $C_{w}([0,T];B^{s}_{p,\infty})$. 
Consequently, we prove the continuous dependence.
	
In conclusion, combining with \textbf{Step 1--Step 4}, we complete the proof of Theorem \ref{th}.
\end{proof}

\section{Non-uniform continuous dependence in both supercritical and critical Besov spaces}
\par
In this section, we aim to study the non-uniform continuous dependence of the Cauchy problem \eqref{fww} in supercritical Besov spaces $B^{s}_{p,r}$ with $s>1+\frac{1}{p},\ 1\leq p\leq+\infty,\ 1\leq r<+\infty$ and critical Besov spaces $B^{1+\frac{1}{p}}_{p,1}$ with $1\leq p<+\infty$.

Before that, we introduce smooth, radial cut-off functions in frequency space. Let $\hat{\psi}\in C_0^\infty(\mathbb{R})$ be an even, real-valued and non-negative function on $\mathcal{D}$ and satisfy
\begin{equation*}
\hat{\psi}(\xi)=\left\{\begin{array}{ll}
1,\quad \text{if}\ |\xi|\leq\frac{1}{4,}\\
0,\quad \text{if}\ |\xi|\geq\frac{1}{2}.
\end{array}\right.
\end{equation*}
From Fourier inversion formula, we can easily deduce that 
\begin{align*}
&\psi(x)=\frac{1}{2\pi}\int\limits_{\mathbb{R}}e^{ix\xi}\hat{\psi}(\xi){\ud}\xi,\\
&\partial_{x}\psi(x)=\frac{1}{2\pi}\int\limits_{\mathbb{R}}i\xi e^{ix\xi}\hat{\psi}(\xi){\ud}\xi,
\end{align*}
which implies by the Fubini theorem that
\begin{align*}
&\|\psi\|_{L^\infty}=\sup\limits_{x\in\mathbb{R}}\frac{1}{2\pi}\bigg|\int\limits_{\mathbb{R}}\cos(x\xi)\hat{\psi}(\xi){\ud}\xi\bigg|\leq\frac{1}{2\pi}\int\limits_{\mathbb{R}}\hat{\psi}(\xi){\ud}\xi\leq C,\\
&\|\partial_{x}\psi\|_{L^\infty}\leq\frac{1}{2\pi}\int\limits_{\mathbb{R}}|\xi|\hat{\psi}(\xi){\ud}\xi\leq C,\\
&\psi(0)=\frac{1}{2\pi}\int\limits_{\mathbb{R}}\hat{\psi}(\xi){\ud}\xi.
\end{align*}

\begin{lemm}\label{psii}
Let $1\leq a\leq\infty$. Then there is a constant $A>0$ such that 
\begin{align}\label{psiii}
\liminf\limits_{n\rightarrow\infty}\bigg\|\psi^2(\cdot)\cos\Big(\frac{33}{24}2^n\cdot\Big)\bigg\|_{L^a}\geq A.
\end{align}
\end{lemm}
\begin{lemm}\label{wsin}
Let $s\in\mathbb{R}$ and $1\leq p\leq\infty,\ 1\leq r<\infty$. Define the high frequency function ${\rm w}^n_0$ by 
\begin{align*}
{\rm w}^n_0(x)=2^{-ns}\psi(x)\sin\Big(\frac{33}{24}2^{n}x\Big),\qquad n\gg 1.
\end{align*}
Then for any $\theta\in\mathbb{R}$, we have 
\begin{align}
&\|{\rm w}^n_0\|_{L^{p}}\leq C2^{-ns}\|\psi\|_{L^{p}}\leq C2^{-ns},\qquad
\|\partial_{x}{\rm w}^n_0\|_{L^{p}}\leq C2^{-ns+n},\label{wn0}\\	
&\|{\rm w}^n_0\|_{B^\theta_{p,r}}\leq C2^{n(\theta-s)}\|\psi\|_{L^p}\leq C2^{n(\theta-s)}.\label{high}
\end{align}	
\end{lemm}
\begin{lemm}\label{vcos}
Let $s\in\mathbb{R}$ and $1\leq p\leq\infty,\ 1\leq r<\infty$. Define the low frequency function ${\rm v}^n_0$ by 
\begin{align*}
	{\rm v}^n_0(x)=\frac{24}{33}2^{-n}\psi(x),\qquad n\gg 1.
\end{align*}
Then 
\begin{align}
&\|{\rm v}^n_0\|_{L^{p}}\leq C2^{-n}\|\psi\|_{L^{p}}\leq C2^{-n},\quad
\|\partial_{x}{\rm v}^n_0\|_{L^{p}}\leq C2^{-n}\|\partial_{x}\psi\|_{L^{p}}\leq C2^{-n},\label{vn0}\\
&\|{\rm v}^n_0\|_{B^{s}_{p,r}}\leq C2^{-s}\|{\rm v}^n_0\|_{L^{p}}\leq C2^{-n-s},\label{vn0s}	
\end{align}
and there is a constant $\tilde{A}>0$ such that  
\begin{align}\label{low}
\liminf\limits_{n\rightarrow\infty}\big\|{\rm v}^n_0\partial_{x}{\rm w}^n_0\big\|_{B^s_{p,\infty}}\geq \tilde{A}.
\end{align}
\end{lemm}
The above lemmas can be proved by a similar way as Lemmas 3.2--3.4 in \cite{lyz} and here we omit it.\\

Define ${\rm w}^n$ by the solution to \eqref{fww} with the initial data ${\rm w}^n_0$. Then we have the following estimates.
\begin{prop}\label{wnw0}
Let $(s,p,r)$ meet the condition \eqref{cond} in Theorem \ref{uniform}.	
Then we have
\begin{align}
&\|{\rm w}^n\|_{B^{s+k}_{p,r}}\leq 2^{kn},\qquad\text{for $k=-1,1,$}\label{wn}\\
&\|{\rm w}^n-{\rm w}^n_0\|_{B^{s}_{p,r}}\leq\left\{\begin{array}{ll}
C2^{-\frac{1}{2}\cdot\frac{1}{p}n},\quad \text{if}\ p<+\infty,\\
C2^{-\frac{\varepsilon_{0}}{2}n},\quad \text{if}\ p=+\infty.
\end{array}\right.\label{w}
\end{align} 
\end{prop}
\begin{proof}
\eqref{high} implies that
\begin{align}
\|{\rm w}^n_0\|_{B^{s+k}_{p,r}}\leq C2^{kn},\ \text{for $k=-1,0,1$.}\label{wn-10}	
\end{align}
From Theorem \ref{th}, we know that there exists a $T=T(\|{\rm w}^n_0\|_{B^s_{p,r}})$ such that \eqref{fww} with initial data ${\rm w}^n_0$ has a unique solution ${\rm w}^n\in E^s_{p,r}(T)$ and $T\approx 1$. Moreover, we known from \eqref{wn-10}
\begin{align}
\|{\rm w}^n\|_{L^{\infty}(0,T;B^s_{p,r})}\leq C \|{\rm w}^n_0\|_{B^s_{p,r}}\leq C.\label{wnbound}
\end{align}
Similar to \eqref{dd}, we get that for $k=\pm 1$,
\begin{align}
\|{\rm w}^{n}(t)\|_{B^{s+k}_{p,r}}
\leq&C\Big(\|{\rm w}^n_0\|_{B^{s+k}_{p,r}}+\int_0^t\|(1-\partial_{xx})^{-1}\partial_{x}{\rm w}^n\|_{B^{s+k}_{p,r}}{\ud}\tau\Big)\notag\\
\leq&C\Big(\|{\rm w}^n_0\|_{B^{s+k}_{p,r}}+\int_0^t\|{\rm w}^n\|_{B^{s+k}_{p,r}}{\ud}\tau\Big)
.\label{wnk}
\end{align}
Combining the Gronwall lemma, \eqref{wn-10} and \eqref{wnk}, we find for all $t\in[0,T]$,
\begin{align}
\|{\rm w}^n(t)\|_{B^{s-1}_{p,r}}\leq C2^{-n}\qquad\text{and}\qquad
\|{\rm w}^n(t)\|_{B^{s+1}_{p,r}}\leq C2^{n}.\label{wn-11}	
\end{align}
Set $\delta={\rm w}^n-{\rm w}^n_0$. Then $\delta$ solves the following problem
\begin{equation*} \left\{\begin{array}{ll}
\partial_{t}\delta+\frac{3}{2}{\rm w}^n\partial_{x}\delta=-\frac{3}{2}\delta\partial_{x}{\rm w}^n_0-\frac{3}{2}{\rm w}^n_0\partial_{x}{\rm w}^n_0+(1-\partial_{xx})^{-1}\partial_{x}\delta+(1-\partial_{xx})^{-1}\partial_{x}{\rm w}^n_0,\\
\delta(0,x)=0.
\end{array}\right.
\end{equation*}
By virtue of Proposition \ref{Besov} and Lemma \ref{product}, we obtain
\begin{align*}
&\|\delta\partial_{x}{\rm w}^n_0\|_{B^{s-1}_{p,r}}\leq C\|\delta\|_{B^{s-1}_{p,r}}\|{\rm w}^n_0\|_{B^{s}_{p,r}},\\
&\|(1-\partial_{xx})^{-1}\partial_{x}\delta\|_{B^{s-1}_{p,r}}\leq C\|\delta\|_{B^{s-2}_{p,r}}\leq C\|\delta\|_{B^{s-1}_{p,r}}.
\end{align*}
\eqref{wn0} and \eqref{high} imply that
\begin{align*}
\|{\rm w}^n_0\partial_{x}{\rm w}^n_0\|_{B^{s-1}_{p,r}}\leq&C\|{\rm w}^n_0\|_{L^\infty}\|\partial_{x}{\rm w}^n_0\|_{B^{s-1}_{p,r}}+\|{\rm w}^n_0\|_{B^{s-1}_{p,r}}\|\partial_{x}{\rm w}^n_0\|_{L^\infty}\\
\leq&C2^{-ns}+2^{-n}2^{-ns+n}\leq C2^{-ns}.
\end{align*}
For $1\leq p<+\infty$, then $0<\frac{1}{p}\leq1$. Hence,
\begin{align*}
\|(1-\partial_{xx})^{-1}\partial_{x}{\rm w}^n_0\|_{B^{s-1}_{p,r}}\leq C\|{\rm w}^n_0\|_{B^{s-\frac{1}{p}}_{p,r}}\leq C\|{\rm w}^n_0\|_{B^{s-1-\frac{1}{p}}_{p,r}}\leq C2^{-(1+\frac{1}{p})n}.	
\end{align*}
For $p=+\infty$, then $s>1+\frac{1}{p}=1$, so there exists a constant $\epsilon_{0}>0$ such that $s>1+\epsilon_{0}$, therefore, we have
\begin{align*}
\|(1-\partial_{xx})^{-1}\partial_{x}{\rm w}^n_0\|_{B^{s-1}_{p,r}}\leq C\|(1-\partial_{xx})^{-1}\partial_{x}{\rm w}^n_0\|_{B^{s-\epsilon_{0}}_{p,r}}\leq C\|{\rm w}^n_0\|_{B^{s-1-\epsilon_{0}}_{p,r}}\leq C2^{-(1+\epsilon_{0})n}.	
\end{align*}
It follows that 
\begin{align*}
\|{\rm w}^n-{\rm w}^n_0\|_{B^{s-1}_{p,r}}=\|\delta\|_{B^{s-1}_{p,r}}\leq\left\{\begin{array}{ll}
	C2^{-(1+\frac{1}{p})n},\quad \text{if}\ p<+\infty,\\
	C2^{-(1+\epsilon_{0})n},\quad \text{if}\ p=+\infty,
\end{array}\right.
\end{align*}
which implies from the interpolation inequality that
\begin{align*}
\|{\rm w}^n-{\rm w}^n_0\|_{B^{s}_{p,r}}\leq\|{\rm w}^n-{\rm w}^n_0\|^{\frac{1}{2}}_{B^{s-1}_{p,r}}\|{\rm w}^n-{\rm w}^n_0\|^{\frac{1}{2}}_{B^{s+1}_{p,r}}\leq\left\{\begin{array}{ll}
	C2^{-\frac{1}{2}\cdot\frac{1}{p}n},\quad \text{if}\ p<+\infty,\\
	C2^{-\frac{\epsilon_{0}}{2}n},\quad \text{if}\ p=+\infty.
\end{array}\right.	
\end{align*}
This thus finish the proof of the proposition.
\end{proof}

In order to obtain the non-uniform continuous dependence for \eqref{fww}, we need to construct a sequence of initial data $u_0^n={\rm w}^n_0+{\rm v}^n_0$.
\begin{prop}
Let $(s,p,r)$ meet the condition \eqref{cond} in Theorem \ref{uniform} and $u^n$ be the solution to \eqref{fww} with initial data $u^n_0$. Define ${\rm z}_0^n=-\frac{3}{2}u^n_0\partial_{x}u^n_0$. Then we have
\begin{align}
\|u^n-u^n_0-t{\rm z}_0^n\|_{B^{s}_{p,r}}\leq Ct^{2}+C2^{-n}.\label{wwn}
\end{align}
\end{prop}
\begin{proof}
Since $u_0^n={\rm w}^n_0+{\rm v}_0^n$, then by Lemmas \ref{wsin}--\ref{vcos} and Proposition \ref{wnw0}, we see
\begin{align}\label{unn}
\|u^n\|_{B^{s+k}_{p,r}}\leq C\|u^n_0\|_{B^{s+k}_{p,r}}\leq C2^{kn},\quad k=0,\pm1.	
\end{align}
Owing to ${\rm z}_0^n=-\frac{3}{2}u^n_0\partial_{x}u_0^n$, we can deduce that 
\begin{align*}
&\|{\rm z}^n_0\|_{B^{s-1}_{p,r}}\leq C\|u^n_0\|_{B^{s-1}_{p,r}}\|\partial_{x}u^n_0\|_{B^{s-1}_{p,r}}\leq C\|u^n_0\|_{B^{s-1}_{p,r}}\|u^n_0\|_{B^{s}_{p,r}}\leq C2^{-n},\\
&\|{\rm z}^n_0\|_{B^{s}_{p,r}}\leq C\|u^n_0\|_{L^\infty}\|u^n_0\|_{B^{s+1}_{p,r}}+\|\partial_{x}u^n_0\|_{L^\infty}\|u^n_0\|_{B^{s}_{p,r}}\leq C2^{-n}2^n+C\leq C,\\
&\|{\rm z}^n_0\|_{B^{s+1}_{p,r}}\leq C\|u^n_0\|_{L^\infty}\|u^n_0\|_{B^{s+2}_{p,r}}+\|\partial_{x}u^n_0\|_{L^\infty}\|u^n_0\|_{B^{s+1}_{p,r}}\leq C2^{-n}2^{2n}+C2^n\leq C2^{n}.	
\end{align*}
Set $u^n=u_0^n+t{\rm z}_0^n+{\rm{e}}^n$. We know that ${\rm{e}}^n$ satisfies
\begin{equation}\label{enn}
\left\{\begin{array}{ll}
\partial_{t}{\rm{e}}^n+\frac{3}{2}u^n\partial_{x}{\rm{e}}^n
=-\frac{3}{2}{\rm{e}}^n\partial_{x}(u^n_0+t{\rm z}^n_0)+(1-\partial_{xx})^{-1}\partial_{x}u^{n}-\frac{3}{2}t\big({\rm z}^n_0\partial_{x}u^n_{0}+u^n_{0}\partial_{x}{\rm z}^n_0\big)-\frac{3}{2}t^{2}{\rm z}^n_0\partial_{x}{\rm z}^n_0,\\
{\rm{e}}^n(0,x)={\rm{e}}^n_{0}(x)=0.
\end{array}\right.
\end{equation}

For $k=-1$, applying Lemma \ref{priori estimate} to \eqref{enn}, we see for all $t\in[0,T]$,
\begin{align}
\|{\rm{e}}^n(t)\|_{B^{s-1}_{p,r}}
\leq&  C\int_0^t\|u^n\|_{B^{s}_{p,r}}\|{\rm{e}}^n(\tau)\|_{B^{s-1}_{p,r}}{\ud}\tau+C\int_0^t\|{\rm{e}}^n\partial_{x}(u^n_0+t{\rm z}^n_0)\|_{B^{s-1}_{p,r}}{\ud}\tau\notag\\
&+C\int_0^t\|(1-\partial_{xx})^{-1}\partial_{x}u^{n}\|_{B^{s-1}_{p,r}}{\ud}\tau+Ct^2\big(\|{\rm z}^n_0\partial_{x}u^n_0\|_{B^{s-1}_{p,r}}+\|u^n_0\partial_{x}{\rm z}^n_0\|_{B^{s-1}_{p,r}}\big)\notag\\
&+Ct^3\|{\rm z}^n_0\partial_{x}{\rm z}^n_0\|_{B^{s-1}_{p,r}}.\label{ensk}
\end{align}
Since $u^n$ is bounded in $C\big([0,T];B^s_{p,r}\big)$ with $T\approx1$, we can deduce from  Lemma \ref{product}, Lemma \ref{wsin} and Lemma \ref{vcos} that
\begin{align*}
&\|{\rm{e}}^n\partial_{x}(u^n_0+t{\rm z}^n_0)\|_{B^{s-1}_{p,r}}\leq\|{\rm{e}}^n\|_{B^{s-1}_{p,r}}(\|u^n_0\|_{B^{s}_{p,r}}+\|{\rm z}^n_0\|_{B^{s}_{p,r}})\leq C\|{\rm{e}}^n\|_{B^{s-1}_{p,r}},\\
&\|(1-\partial_{xx})^{-1}\partial_{x}u^{n}\|_{B^{s-1}_{p,r}}\leq C\|u^{n}\|_{B^{s-2}_{p,r}}\leq C\|u^{n}\|_{B^{s-1}_{p,r}}\leq C\|u^{n}_{0}\|_{B^{s-1}_{p,r}}\leq C2^{-n},\\
&\|{\rm z}^n_0\partial_{x}u^n_0\|_{B^{s-1}_{p,r}}\leq \|{\rm z}^n_0\|_{B^{s-1}_{p,r}}\|u^n_0\|_{B^{s}_{p,r}}\leq C2^{-n},\\
&\|u^n_0\partial_{x}{\rm z}^n_0\|_{B^{s-1}_{p,r}}\leq \|u^n_0\|_{B^{s-1}_{p,r}}\|{\rm z}^n_0\|_{B^{s}_{p,r}}\leq C2^{-n},\\
&\|{\rm z}^n_0\partial_{x}{\rm z}^n_0\|_{B^{s-1}_{p,r}}\leq C\|{\rm z}^n_0\|_{B^{s-1}_{p,r}}\|{\rm z}^n_0\|_{B^{s}_{p,r}}\leq C2^{-n}.
\end{align*}
Plugging the above inequaties into \eqref{ensk}, we can find 
\begin{align}
\|{\rm{e}}^n(t)\|_{B^{s-1}_{p,r}}\leq& C\int_0^t\|{\rm{e}}^n(\tau)\|_{B^{s-1}_{p,r}}{\ud}\tau+Ct2^{-n}\leq Ct2^{-n}\label{enk-1}
\end{align}
where we use the Gronwall lemma in the last inequality.

For $k=0$, we have
\begin{align}
\|{\rm{e}}^n(t)\|_{B^{s}_{p,r}}
\leq&  C\int_0^t\|{\rm{e}}^n(\tau)\|_{B^{s}_{p,r}}\|u^n\|_{B^{s}_{p,r}}{\ud}\tau+C\int_0^t\|{\rm{e}}^n\partial_{x}(u^n_0+t{\rm z}^n_0)\|_{B^{s}_{p,r}}{\ud}\tau\notag\\
&+C\int_0^t\|(1-\partial_{xx})^{-1}\partial_{x}u^{n}\|_{B^{s}_{p,r}}{\ud}\tau+Ct^2\big(\|{\rm z}^n_0\partial_{x}u^n_0\|_{B^{s}_{p,r}}+\|u^n_0\partial_{x}{\rm z}^n_0\|_{B^{s}_{p,r}}\big)\notag\\
&+Ct^3\|{\rm z}^n_0\partial_{x}{\rm z}^n_0\|_{B^{s}_{p,r}}.\label{ensk1}
\end{align}
Similarly, according to Lemma \ref{product}, Lemma \ref{wsin} and Lemma \ref{vcos}, one has 
\begin{align*}
&\|{\rm{e}}^n\partial_{x}(u^n_0+t{\rm z}^n_0)\|_{B^{s}_{p,r}}
\leq C\|{\rm{e}}^n\|_{B^{s}_{p,r}}\|u^n_0,{\rm z}^n_0\|_{B^{s}_{p,r}}+C\|{\rm{e}}^n\|_{B^{s-1}_{p,r}}\|u^n_0,{\rm z}^n_0\|_{B^{s+1}_{p,r}}\leq C\|{\rm{e}}^n\|_{B^{s}_{p,r}}+C2^{n}\|{\rm{e}}^n\|_{B^{s-1}_{p,r}},\\
&\|(1-\partial_{xx})^{-1}\partial_{x}u^{n}\|_{B^{s}_{p,r}}\leq C\|u^{n}\|_{B^{s-1}_{p,r}}\leq  C\|u^{n}_{0}\|_{B^{s-1}_{p,r}}\leq C2^{-n},\\
&\|{\rm z}^n_0\partial_{x}u^n_0\|_{B^{s}_{p,r}}\leq C\Big(\|{\rm z}^n_0\|_{B^{s}_{p,r}}\|u^n_0\|_{B^{s}_{p,r}}+\|{\rm z}^n_0\|_{B^{s-1}_{p,r}}\|u^n_0\|_{B^{s+1}_{p,r}}\Big)\leq C,\\
&\|u^n_0\partial_{x}{\rm z}^n_0\|_{B^{s}_{p,r}}\leq C\Big(\|u^n_0\|_{B^{s}_{p,r}}\|{\rm z}^n_0\|_{B^{s}_{p,r}}+\|u^n_0\|_{B^{s-1}_{p,r}}\|{\rm z}^n_0\|_{B^{s+1}_{p,r}}\Big)\leq C,\\
&\|{\rm z}^n_0\partial_{x}{\rm z}^n_0\|_{B^{s}_{p,r}}\leq C\Big(\|{\rm z}^n_0\|_{B^{s}_{p,r}}\|{\rm z}^n_0\|_{B^{s}_{p,r}}+\|{\rm z}^n_0\|_{B^{s-1}_{p,r}}\|{\rm z}^n_0\|_{B^{s+1}_{p,r}}\Big)\leq C.
\end{align*}
\eqref{enk-1} togethering with  \eqref{ensk1} yield
\begin{align*}
\|{\rm{e}}^n(t)\|_{B^{s}_{p,r}}\leq&C\int_0^t2^n\|{\rm{e}}^n(\tau)\|_{B^{s-1}_{p,r}}{\ud}\tau+\int_0^t\|{\rm{e}}^n(\tau)\|_{B^{s}_{p,r}}{\ud}\tau+Ct^{2}+Ct2^{-n}\\
\leq&Ct^{2}+Ct2^{-n}+\int_0^t\|{\rm{e}}^n(\tau)\|_{B^{s}_{p,r}}{\ud}\tau\\
\leq&Ct^{2}+C2^{-n}.
\end{align*}
Thus, we prove the proposition.  
\end{proof}

Now we give the proof of Theorem \ref{uniform}. 
\begin{proof}[{\rm\textbf{The proof of Theorem \ref{uniform}:}}]
Lemma \ref{vcos} guarantees that
\begin{align*}
\|u^n_0-{\rm w}^n_0\|_{B^{s}_{p,r}}=\|{\rm v}^n_0\|_{B^{s}_{p,r}}\leq C2^{-n}.
\end{align*}
It follows that 
\begin{align*}
\lim\limits_{n\rightarrow\infty}\Big\|u^n_0-{\rm w}^n_0\Big\|_{B^{s}_{p,r}}=0.
\end{align*}
Moreover, we have from \eqref{vn0s}, \eqref{w} and \eqref{wwn}
\begin{align*}
\|u^n-{\rm w}^n\|_{B^{s}_{p,r}}=&\|t{\rm z}^n_0+{\rm v}^n_0+{\rm w}^n_0-{\rm w}^n+{\rm{e}}^n\|_{B^{s}_{p,r}}\\
\geq& t\|{\rm z}^n_0\|_{B^{s}_{p,r}}-\|{\rm v}^n_0\|_{B^{s}_{p,r}}-\|{\rm w}^n-{\rm w}^n_0\|_{B^{s}_{p,r}}-\|{\rm{e}}^n\|_{B^{s}_{p,r}}\\
\geq&\left\{\begin{array}{ll}
	t\|{\rm z}^n_0\|_{B^{s}_{p,r}}-Ct^2-C2^{-\frac{1}{2}\cdot\frac{1}{p}n},\quad \text{if}\ p<+\infty,\\
	t\|{\rm z}^n_0\|_{B^{s}_{p,r}}-Ct^2-C2^{-\frac{\epsilon_{0}}{2}n},\quad \text{if}\ p=+\infty.
\end{array}\right.	
\end{align*}
Owing to ${\rm z}_0^n=-\frac{3}{2}\big[{\rm v}_0^n\partial_{x}{\rm w}_0^n+{\rm w}_0^n\partial_{x}{\rm w}_0^n+{\rm w}_0^n\partial_{x}{\rm v}_0^n+{\rm v}_0^n\partial_{x}{\rm v}_0^n\big]$, we find from Lemmas \ref{wsin}--\ref{vcos}
\begin{align*}
&\|{\rm w}_0^n\partial_{x}{\rm w}_0^n\|_{B^{s}_{p,r}}\leq \|{\rm w}_0^n\|_{L^\infty}\|{\rm w}_0^n\|_{B^{s+1}_{p,r}}+\|\partial_{x}{\rm w}_0^n\|_{L^\infty}\|{\rm w}_0^n\|_{B^{s}_{p,r}}\leq C2^{-n(s-1)}\leq\left\{\begin{array}{ll}
	C2^{-\frac{1}{2}\cdot\frac{1}{p}n},\quad \text{if}\ p<+\infty,\\
	C2^{-\frac{\epsilon_{0}}{2}n},\quad \text{if}\ p=+\infty,
\end{array}\right.\\
&\|{\rm w}_0^n\partial_{x}{\rm v}_0^n\|_{B^{s}_{p,r}}\leq\|{\rm w}_0^n\|_{B^{s}_{p,r}}\|{\rm v}_0^n\|_{B^{s+1}_{p,r}}\leq C2^{-n},\\
&\|{\rm v}_0^n\partial_{x}{\rm v}_0^n\|_{B^{s}_{p,r}}\leq\|{\rm v}_0^n\|_{B^{s}_{p,r}}\|{\rm v}_0^n\|_{B^{s+1}_{p,r}}\leq C2^{-n}.
\end{align*}
By means of the embedding $B^{s}_{p,r}\hookrightarrow B^{s}_{p,\infty}$, we get 
\begin{align*}
-\|{\rm w}_0^n\partial_{x}{\rm w}_0^n\|_{B^{s}_{p,\infty}}-\|{\rm w}_0^n\partial_{x}{\rm v}_0^n\|_{B^{s}_{p,\infty}}-\|{\rm v}_0^n\partial_{x}{\rm v}_0^n\|_{B^{s}_{p,\infty}}\geq& -\|{\rm w}_0^n\partial_{x}{\rm w}_0^n\|_{B^{s}_{p,r}}-\|{\rm w}_0^n\partial_{x}{\rm v}_0^n\|_{B^{s}_{p,r}}-\|{\rm v}_0^n\partial_{x}{\rm v}_0^n\|_{B^{s}_{p,r}}\\
\geq&\left\{\begin{array}{ll}
	-C2^{-\frac{1}{2}\cdot\frac{1}{p}n},\quad \text{if}\ p<+\infty,\\
    -C2^{-\frac{\epsilon_{0}}{2}n},\quad \text{if}\ p=+\infty.
\end{array}\right.
\end{align*}
Hence, for $t>0$ small enough, we have
\begin{align*}
\|u^n-{\rm w}^n\|_{B^{s}_{p,r}}
\geq\left\{\begin{array}{ll}
	t\|{\rm v}^n_0\partial_{x}{\rm w}_0^n\|_{B^{s}_{p,\infty}}-C2^{-\frac{1}{2}\cdot\frac{1}{p}n},\quad \text{if}\ p<+\infty,\\
	t\|{\rm v}^n_0\partial_{x}{\rm w}_0^n\|_{B^{s}_{p,\infty}}-C2^{-\frac{\epsilon_{0}}{2}n},\quad \text{if}\ p=+\infty.
\end{array}\right.
\end{align*}
Combining with \eqref{low}, we finally obtain
\begin{align*}
\liminf\limits_{n\rightarrow\infty}\Big\|u^n-{\rm w}^n\Big\|_{B^{s}_{p,r}}\gtrsim t,	
\quad\text{for $t>0$ small enough.}
\end{align*}
This proves Theorem \ref{uniform}.
\end{proof}

\section{Ill-posedness in Besov spaces}
\par
This section is devoted to investigating the ill-posedness for the Cauchy problem \eqref{fww} in Besov space $B^{\sigma}_{p,\infty}$ with $ \sigma>3+\frac{1}{p},\ 1\leq p\leq+\infty$.

Similar to \cite{lyz}, we can verify that for $g_n(x):=\psi(x)\cos \big(\frac{33}{24}2^{n}x\big)$ and $n\geq 2$,
\begin{numcases}{\Delta_j(g_n)=}
g_n, &\text{if} $j=n$,\nonumber\\
0, &\text{if} $j\neq n$.\label{fn}
\end{numcases}
We can also obtain the following similar result:
\begin{lemm}\label{le4}
Let $4\leq l,\ n\in \mathbb{N}^+$. Define the function $h^{l}_{m,n}(x)$ by
$$h^{l}_{m,n}(x):=\psi(x)\cos \Big(\frac{33}{24}\big(2^{ln}\pm2^{lm}\big)x\Big)\quad\text{with}\quad 0\leq m\leq n-1.$$
Then we have
\begin{numcases}
{\Delta_j(h^l_{m,n})=}
h^l_{m,n}, &if $j=ln$,\nonumber\\
0, &if $j\neq ln$.\nonumber
\end{numcases}
\end{lemm}
\begin{proof}
The proof is similar to that of in \cite{lyz2}, and here we omit it.
\end{proof}
\begin{lemm}\label{le5}
Let $1\leq p\leq +\infty.$ Define the initial data $u_0(x)$ as
\begin{align*}
u_0(x):=\sum\limits^{\infty}_{n=0}2^{-ln\sigma}g^l_n(x),
\end{align*}
where
$$g^l_n(x):=\psi(x)\cos \big(\frac{33}{24}2^{ln}x\big),\quad n\geq 0.$$
Then for any $\sigma>3+\frac 1 p$ and for some $l$ large enough, we have
\begin{align*}
&\|u_0\|_{B^{\sigma}_{p,\infty}}\leq C,\\
&\|\Delta_{ln}\big(u_0^2\big)\|_{L^p}\geq c2^{-ln\sigma}.
\end{align*}
\end{lemm}
\begin{proof}
Appealing to the definition of Besov spaces, the support of $\varphi(2^{-j}\cdot)$ and \eqref{fn}, we see
\begin{align*}
\|u_0\|_{B^{\sigma}_{p,\infty}}&=\sup_{j\geq -1}2^{j\sigma}\|\Delta_{j}u_0\|_{L^p}\\
&=\sup_{j\geq 0}\big\|\psi(x)\cos \big(\frac{33}{24}2^{j}x\big)\big\|_{L^p}\\
&\leq C.
\end{align*}
Notice that the simple fact
$$\cos(\mathbf{a}+\mathbf{b})+\cos(\mathbf{a}-\mathbf{b})=2\cos \mathbf{a} \cos \mathbf{b}$$
and
$$\sum^{\infty}_{n=0}\sum^{\infty}_{m=0,m\neq n}\mathbf{X}_n\mathbf{X}_m=2\sum^{\infty}_{n=0}\sum^{n-1}_{m=0}\mathbf{X}_n\mathbf{X}_m,$$
then direct computations give
\begin{align*}
u_0^2(x)
=&\frac12\sum^{\infty}_{n=0}2^{-2ln\sigma}\psi^2(x)
+\frac12\sum^{\infty}_{n=0}2^{-2ln\sigma}\psi^2(x)\cos\big(\frac{33}{24}2^{ln+1}x\big)\\
&+\sum^{\infty}_{n=1}\sum^{n-1}_{m=0}2^{-l(n+m)\sigma}\psi^2(x)\Big[
\cos\big(\frac{33}{24}(2^{ln}-2^{lm})x\big)+\cos\big(\frac{33}{24}(2^{ln}+2^{lm})x\big)\Big].
\end{align*}
Lemma \ref{le4} yields
\begin{align*}
\Delta_{ln}\big(u_0^2\big)
=&2^{-ln\sigma}\psi^2(x)\Big[
\cos\big(\frac{33}{24}(2^{ln}-1)x\big)+
\cos\big(\frac{33}{24}(2^{ln}+1)x\big)\Big]\\ &+\sum^{n-1}_{m=1}2^{-l(n+m)\sigma}\psi^2(x)\Big[\cos\big(\frac{33}{24}(2^{ln}-2^{lm})x\big)+\cos\big(\frac{33}{24}(2^{ln}+2^{lm})x\big)\Big]\\
:=&\mathrm{K}_1+\mathrm{K}_2,
\end{align*}
where we denote
\begin{align*}
&\mathrm{K}_1=2\cdot2^{-ln\sigma}\psi^2(x)\cos\big(\frac{33}{24}2^{ln}x\big)\cos\big(\frac{33}{24}x\big),\\
&\mathrm{K}_2=2\sum^{n-1}_{m=1}2^{-l(n+m)\sigma}\psi^2(x)\cos\big(\frac{33}{24}2^{ln}x\big)\cos\big(\frac{33}{24}2^{lm}x\big).
\end{align*}
For the first term $\mathrm{K}_1$, after a simple calculation, we discover
\begin{align}\label{l0}
\|\mathrm{K}_1\|_{L^p}&\geq 2^{-ln\sigma}\big\|\psi^2(x)\cos\big(\frac{33}{24}2^{ln}x\big)\cos\big(\frac{33}{24}x\big)\big\|_{L^p}.
\end{align}
Similar to Lemma 3.2 in \cite{lyz}, we have for some $\lambda>0$
\begin{align}\label{l2}
\big\|\psi^2(x)\cos\big(\frac{33}{24}2^{ln}x\big)\cos\big(\frac{33}{24}x\big)\big\|_{L^p}&\geq c\big(p,\lambda,\psi(0)\big).
\end{align}
Then we obtain from \eqref{l0}--\eqref{l2} that
\begin{align}\label{l3}
\|\mathrm{K}_1\|_{L^p}&\geq c2^{-ln\sigma}.
\end{align}
For the second term $\mathrm{K}_2$, from a straightforward calculation, we deduce
\begin{align}\label{l4}
\|\mathrm{K}_2\|_{L^p}\leq C\sum^{n-1}_{m=1}2^{-l(n+m)\sigma}\leq C2^{-l(n+1)\sigma}.
\end{align}
\eqref{l3} and \eqref{l4} together yield that
\begin{align*}
\|\Delta_{ln}\big(u_0^2\big)\|_{L^p}\geq(c-C2^{-l\sigma})2^{-ln\sigma}.
\end{align*}
Choosing $l\geq4$ such that $c-C2^{-l\sigma}>0$ and then we finish the proof of Lemma \ref{le5}.
\end{proof}

\begin{prop}\label{pro3.1}
Let $s=\sigma-2$ and $u_0\in B^{\sigma}_{p,\infty}$. Assume $u\in L^\infty_TB^{\sigma}_{p,\infty}$ be the solution to the Cauchy problem \eqref{fww}, we have
\begin{align*}
&\|u(t)-u_0\|_{B^{s-1}_{p,\infty}}\leq Ct\Big(\|u_0\|_{B^{s-1}_{p,\infty}}+\|u_0\|_{B^{s-1}_{p,\infty}}\|u_0\|_{B^{s}_{p,\infty}}\Big),\\
&\|u(t)-u_0\|_{B^{s}_{p,\infty}}\leq Ct\big(\|u_0\|_{B^{s-1}_{p,\infty}}+\|u_0\|^2_{B^{s}_{p,\infty}}+\|u_0\|_{B^{s-1}_{p,\infty}}\|u_0\|_{B^{s+1}_{p,\infty}}\big),\\
&\|u(t)-u_0\|_{B^{s+1}_{p,\infty}}\leq Ct\big(\|u_0\|_{B^{s}_{p,\infty}}+\|u_0\|_{B^{s}_{p,\infty}}\|u_0\|_{B^{s+1}_{p,\infty}}+\|u_0\|_{B^{s-1}_{p,\infty}}\|u_0\|_{B^{s+2}_{p,\infty}}\big).
\end{align*}
\end{prop}
\begin{proof}
For $\gamma>0$, taking advantage of Lemma \ref{priori estimate} and Theorem \ref{th}, we find
\begin{align}\label{s1}
\|u(t)\|_{L^\infty_TB^\gamma_{p,\infty}}\leq C\|u_0\|_{B^\gamma_{p,\infty}}.
\end{align}
By the Mean Value Theorem, we obtain from \eqref{s1} that
\begin{align*}
\|u(t)-u_0\|_{B^s_{p,\infty}}
\leq&\int^t_0\|\partial_{\tau} u\|_{B^s_{p,\infty}} {\ud}\tau\\
\leq&\int^t_0\|(1-\partial_{xx})^{-1}\partial_{x}u\|_{B^s_{p,\infty}} {\ud}\tau+\frac{3}{2}\int^t_0\|uu_{x}\|_{B^s_{p,\infty}} {\ud}\tau\\
\leq& Ct\big(\|u\|_{L_t^\infty B^{s-1}_{p,\infty}}+\|u\|^{2}_{L_t^\infty B^{s}_{p,\infty}}+\|u\|_{L_t^\infty L^\infty}\|u_{x}\|_{L_t^\infty B^{s}_{p,\infty}}\big)\\
\leq& Ct\big(\|u\|_{L_t^\infty B^{s-1}_{p,\infty}}+\|u\|^{2}_{L_t^\infty B^{s}_{p,\infty}}+\|u\|_{L_t^\infty B^{s-1}_{p,\infty}}\|u\|_{L_t^\infty B^{s+1}_{p,\infty}}\big)\\
\leq&Ct\big(\|u_0\|_{B^{s-1}_{p,\infty}}+\|u_0\|^2_{B^{s}_{p,\infty}}+\|u_0\|_{B^{s-1}_{p,\infty}}\|u_0\|_{B^{s+1}_{p,\infty}}\big).
\end{align*}	
Following the same arguments, we get
\begin{align*}
\|u(t)-u_0\|_{B^{s-1}_{p,\infty}}
\leq& \int^t_0\|\partial_{\tau} u\|_{B^{s-1}_{p,\infty}} {\ud}\tau\\
\leq&\int^t_0\|(1-\partial_{xx})^{-1}\partial_{x}u\|_{B^{s-1}_{p,\infty}} {\ud}\tau+ \frac{3}{2}\int^t_0\|u u_{x}\|_{B^{s-1}_{p,\infty}} {\ud}\tau\\
\leq&Ct\Big(\|u\|_{L_t^\infty B^{s-1}_{p,\infty}}+\|u\|_{L_t^\infty B^{s-1}_{p,\infty}}\|u\|_{L_t^\infty B^{s}_{p,\infty}}\Big)\\
\leq&Ct\Big(\|u_0\|_{B^{s-1}_{p,\infty}}+\|u_0\|_{B^{s-1}_{p,\infty}}\|u_0\|_{B^{s}_{p,\infty}}\Big)
\end{align*}
and
\begin{align*}
\|u(t)-u_0\|_{B^{s+1}_{p,\infty}}
\leq& \int^t_0\|\partial_{\tau} u\|_{B^{s+1}_{p,\infty}} {\ud}\tau\\
\leq&\int^t_0\|(1-\partial_{xx})^{-1}\partial_{x}u\|_{B^{s+1}_{p,\infty}}{\ud}\tau+\frac{3}{2}\int^t_0\|uu_{x}\|_{B^{s+1}_{p,\infty}} {\ud}\tau\\
\leq& Ct\big(\|u\|_{L_t^\infty B^{s}_{p,\infty}}+\|u\|_{L_t^\infty B^{s}_{p,\infty}}\|u\|_{L_t^\infty B^{s+1}_{p,\infty}}+\|u\|_{L_t^\infty B^{s-1}_{p,\infty}}\|u\|_{L_t^\infty B^{s+2}_{p,\infty}}\big)\\
\leq&Ct\big(\|u_0\|_{B^{s}_{p,\infty}}+\|u_0\|_{B^{s}_{p,\infty}}\|u_0\|_{B^{s+1}_{p,\infty}}+\|u_0\|_{B^{s-1}_{p,\infty}}\|u_0\|_{B^{s+2}_{p,\infty}}\big).
\end{align*}
Thus, we finish the proof of Proposition \ref{pro3.1}.
\end{proof}

\begin{prop}\label{pro3.2}
Let $s=\sigma-2$ and $u_0\in B^{\sigma}_{p,\infty}$. Assume that $u\in L^\infty_TB^{\sigma}_{p,\infty}$ be the solution to the Cauchy problem \eqref{fww}, we have
\begin{align*}
\|\mathbf{w}(t,u_0)\|_{B^{s}_{p,\infty}}
\leq&Ct^2\Big(\|u_0\|_{B^{s-1}_{p,\infty}}+\|u_0\|_{B^{s-1}_{p,\infty}}\|u_0\|_{B^{s}_{p,\infty}}+\|u_0\|^3_{B^s_{p,\infty}}+\|u_0\|_{B^{s-1}_{p,\infty}}\|u_0\|_{B^{s}_{p,\infty}}\|u_0\|_{B^{s+1}_{p,\infty}}\\
&\qquad+\|u_0\|_{B^{s-1}_{p,\infty}}^{2}\|u_0\|_{B^{s+2}_{p,\infty}}\Big),
\end{align*}
here and in what follows we denote
\begin{align*}
&\mathbf{w}(t,u_0):=u(t)-u_0-t\mathbf{T}(u_0),\\
&\mathbf{T}(u_0):=-\frac{3}{2}u_0\partial_{x} u_0+(1-\partial_{xx})^{-1}\partial_{x} u_0.
\end{align*}
In particular, we obtain
\begin{align*}
\|\mathbf{w}(t,u_0)\|_{B^{\sigma-2}_{p,\infty}}\leq C\big(\|u_0\|_{B^{\sigma}_{p,\infty}}\big)t^{2}.
\end{align*}
\end{prop}
\begin{proof}
Taking advantage of the Mean Value Theorem and Eq. \eqref{fww}, and then using Lemma \ref{product}, we see
\begin{align*}
\|\mathbf{w}(t,u_0)\|_{B^s_{p,\infty}}
\leq&\int^t_0\|\partial_{\tau}u-\mathbf{T}(u_0)\|_{B^s_{p,\infty}} {\ud}\tau\\
\leq&\int^t_0\|(1-\partial_{xx})^{-1}\partial_{x}(u-u_{0})\|_{B^s_{p,\infty}}{\ud}\tau+\frac{3}{2}\int^t_0\|u\partial_{x}u-u_0\partial_{x}u_0\|_{B^s_{p,\infty}} {\ud}\tau\\
\leq&C\int^t_0\|u(\tau)-u_0\|_{B^{s-1}_{p,\infty}}{\ud}\tau+C\int^t_0\|u(\tau)-u_0\|_{B^{s}_{p,\infty}}\|u_0\|_{B^s_{p,\infty}}{\ud}\tau\\
&+C\int^t_0\|u(\tau)-u_0\|_{B^{s-1}_{p,\infty}}\|u(\tau)\|_{B^{s+1}_{p,\infty}} {\ud}\tau+C\int^t_0\|u(\tau)-u_0\|_{B^{s+1}_{p,\infty}}\|u_0\|_{B^{s-1}_{p,\infty}}{\ud}\tau\\
\leq&Ct^2\Big(\|u_0\|_{B^{s-1}_{p,\infty}}+\|u_0\|_{B^{s-1}_{p,\infty}}\|u_0\|_{B^{s}_{p,\infty}}+\|u_0\|^3_{B^s_{p,\infty}}+\|u_0\|_{B^{s-1}_{p,\infty}}\|u_0\|_{B^{s}_{p,\infty}}\|u_0\|_{B^{s+1}_{p,\infty}}\\
&\qquad+\|u_0\|_{B^{s-1}_{p,\infty}}^{2}\|u_0\|_{B^{s+2}_{p,\infty}}\Big),
\end{align*}
where we have used Proposition \ref{pro3.1} in the last step.
	
Thus, we complete the proof of Proposition \ref{pro3.2}.
\end{proof}

Now we present the proof of Theorem \ref{ill}.
\begin{proof}[\bf The proof of Theorem \ref{ill}:]
Using Proposition \ref{bern}, Lemma \ref{le5} and Proposition \ref{pro3.2}, we get
\begin{align*}
\|u(t)-u_0\|_{B^\sigma_{p,\infty}}
\geq&2^{{ln\sigma}}\big\|\Delta_{ln}\big(u(t)-u_0\big)\big\|_{L^p}=2^{{ln\sigma}}\big\|\Delta_{ln}\big(t\mathbf{T}(u_0)+\mathbf{w}(t,u_0)\big)\big\|_{L^p}\\
\geq& t2^{{ln\sigma}}\|\Delta_{ln}\big(\mathbf{T}(u_0)\big)\|_{L^p}-2^{{2ln}}2^{{ln(\sigma-2)}}\big\|\Delta_{ln}\big(\mathbf{w}(t,u_0)\big)\big\|_{L^p}\\
\geq& ct2^{{ln}(\sigma+1)}\|\Delta_{ln}\big(u_0^2\big)\|_{L^p}-Ct\|u_0\|_{B^{\sigma-1}_{p,\infty}}-C2^{2{ln}}\|\mathbf{w}(t,u_0)\|_{B^{\sigma-2}_{p,\infty}}\\
\geq& ct2^{{ln}(\sigma+1)}\|\Delta_{ln}\big(u_0^2\big)\|_{L^p}-Ct-C2^{2{ln}}t^2\\
\geq& ct2^{{ln}}-Ct-C2^{2{ln}}t^2.
\end{align*}
Then, for $l\geq 4$, taking $n>N$ large  enough such that $c2^{{ln}}\geq 2C$, we deduce that
\begin{align*}
\|u(t)-u_0\|_{B^\sigma_{p,\infty}}\geq ct2^{{ln}}-C2^{2{ln}}t^2.
\end{align*}
Thus, choosing $t2^{ln}\approx\ep$ with small $\ep$, we eventually conclude that
\begin{align*}
\|u(t)-u_0\|_{B^\sigma_{p,\infty}}\geq c\ep-C\ep^2\geq c_1\ep.
\end{align*}
This proves Theorem \ref{ill}.
\end{proof}

\noindent\textbf{Acknowledgements.}
Y. Guo was supported by the Guangdong Basic and Applied Basic Research Foundation (No. 2020A1515111092) and Research Fund of Guangdong-Hong Kong-Macao Joint Laboratory for Intelligent Micro-Nano Optoelectronic Technology (No. 2020B1212030010).


\end{document}